\documentclass[11pt,
a4paper]{article}

\usepackage[top=27mm, bottom=27mm, left=22mm, right=22mm]{geometry}
\usepackage{amsfonts}
\usepackage{mathrsfs}
\usepackage{comment}
\usepackage{amsmath}
\usepackage{amssymb}
\usepackage{amsthm}
\usepackage{amscd}
\usepackage{graphicx}
\usepackage{indentfirst}
\usepackage[all]{xy}
\usepackage{titlesec}
\usepackage{enumerate}
\usepackage{bm}
\usepackage{enumitem}
\usepackage{color}
\usepackage{dsfont}
\usepackage{arydshln}
\newtheorem{theorem}{Theorem}[section]
\newtheorem{lemma}[theorem]{Lemma}

\newtheorem{conjecture}[theorem]{Conjecture}

\newtheorem{claim}[theorem]{Claim}

\newtheorem{fact}[theorem]{Fact}

\newcommand{\ma}{\mathcal}

\newcommand{\mr}{\mathscr}
\newcommand{\s}{\subseteq}

\newcommand{\fr}{\frac}
\newcommand{\lc}{\lceil}
\newcommand{\rc}{\rceil}
\newcommand{\lf}{\lfloor}
\newcommand{\rf}{\rfloor}

\begin{document}

\title{New Tur\'an exponents for two extremal hypergraph problems}
\author{Chong Shangguan\footnote{Research Center for Mathematics and Interdisciplinary Sciences, Shandong University,
Qingdao 266237, China. Email address: theoreming@163.com. This work was accomplished while the author was a postdoc at Tel Aviv University.}  and Itzhak Tamo\footnote{Department of Electrical Engineering - Systems, Tel Aviv University, Tel Aviv 6997801, Israel. Email address: zactamo@gmail.com.}}
\date{}
\maketitle

\begin{abstract}
An $r$-uniform hypergraph is called $t$-cancellative if for any $t+2$ distinct edges $A_1,\ldots,A_t,B,C$, it holds that $(\cup_{i=1}^t A_i)\cup B\neq (\cup_{i=1}^t A_i)\cup C$.
It is called $t$-union-free if for any two distinct subsets $\ma{A},\ma{B}$, each consisting of at most $t$ edges, it holds that $\cup_{A\in\ma{A}} A\neq \cup_{B\in\ma{B}} B$. Let $C_t(n,r)$ (resp. $U_t(n,r)$) denote the maximum number of edges of a $t$-cancellative (resp. $t$-union-free) $r$-uniform hypergraph on $n$ vertices. %For fixed $r\ge 3,t\ge 3$ and $n\rightarrow\infty$, previously it was known that $$\Omega(n^{\lc\fr{r}{t+1}\rc})=C_t(n,r)=O(n^{\lc\fr{r}{\lf t/2\rf}\rc}) \text{ \quad and\quad } \Omega(n^{\lc\fr{r}{t}\rc})=U_t(n,r)=O(n^{\lc\fr{r}{t-1}\rc}).$$
Among other results, we show that for fixed $r\ge 3,t\ge 3$ and $n\rightarrow\infty$
$$\Omega(n^{\lf\frac{2r}{t+2}\rf+\frac{2r\pmod{t+2}}{t+1}})=C_t(n,r)=O(n^{\lceil\frac{r}{\lf t/2\rf+1}\rceil})\text{ and } \Omega(n^{\fr{r}{t-1}})=U_t(n,r)=O(n^{\lc\fr{r}{t-1}\rc}),$$
thereby significantly narrowing the gap between the previously known lower and upper bounds.
%in the exponents of $n$.
In particular, we determine the Tur\'an exponent of $C_t(n,r)$ when $2\mid t \text{ and } (t/2+1)\mid r$, and of $U_t(n,r)$ when $(t-1)\mid r$.

The main tool used in proving the two lower bounds is a novel connection between these problems and  sparse hypergraphs.
%An $r$-uniform hypergraph is called $\mr{G}_r(v,e)$-free if the union of any $e$ distinct edges contains at least $v+1$ vertices. The main ingredients of  are the following interesting observations:
\end{abstract}

\section{Introduction}

 A hypergraph $\ma{H}$ on {\it vertex} set $[n]:=\{1,\ldots,n\}$ is simply a family of distinct subsets of $[n]$,  called {\it edges} of $\ma{H}$. If each edge is of fixed size $r$, then $\ma{H}$ is said to be  {\it $r$-uniform} or an {\it $r$-graph}. For a positive integer $t$, $\ma{H}$ is called {\it $t$-cancellative} if for any $t+2$ distinct edges $A_1,\ldots,A_t,B,C\in\ma{H}$, it holds that
$$(\cup_{i=1}^t A_i)\cup B\neq (\cup_{i=1}^t A_i)\cup C.$$
Furthermore, $\ma{H}$ is called {\it $t$-union-free} if for any two distinct subsets of edges $\ma{A},\ma{B}\s\ma{H}$, with $1\le |\ma{A}|,|\ma{B}|\le t$, it holds that
$$\cup_{A\in\ma{A}} A\neq \cup_{B\in\ma{B}} B.$$

\noindent In this  paper we study the maximum size (maximum number of edges) of an $r$-graph that is either {\it $t$-cancellative} or {\it $t$-union-free}.
%that satisfy one of the two properties mentioned above.

To the best of our knowledge (see also \cite{Frankl-Furedi-1,Furedi-canc}), the study of 1-cancellative  and 2-union-free hypergraphs (not necessarily uniform) were initiated by Erd\H{o}s and Katona \cite{Erdos-Katona} and Erd\H{o}s and Moser \cite{Erdos-Moser}, respectively. Notice that the question of $1$-union-free is trivial, since any hypergraph is $1$-union-free. The study of the closely related problem of {\it weakly union-free families} dates back to the work of Erd\H{o}s \cite{Erdos-weak} in 1938. Later, the general questions of  $t$-cancellative and $t$-union-free hypergraphs were first considered by F\"uredi \cite{Furedi-canc} (for 2-cancellative hypergraphs, see also \cite{Korner-2-canc}) and Kautz and Singleton \cite{kautz1964nonrandom}, respectively.

Not much is known about these problems for $r$-graphs and $t\ge 3$, besides this work. In fact, most of the known results were derived by translating results from the extensive literature on cover-free hypergraphs to results on cancellative and union-free hypergraphs. Next, we  define cover-free hypergraphs and present the  simple observations that facilitate the translation of the results.

An $r$-graph $\ma{H}$ is called {\it $t$-cover-free} by Erd\H{o}s, Frankl and F\"uredi \cite{Erdos-cff} (see also \cite{kautz1964nonrandom} for a different terminology), if for any $t+1$ distinct edges $A_1,\ldots,A_t,B\in\ma{H}$, it holds that $$B\nsubseteq \cup_{i=1}^t A_i.$$ In the literature, the following observations are well-known.

\begin{itemize}
  \item [$(a)$] If $\ma{H}$ is $(t+1)$-cover-free then it is also $t$-cancellative; if $\ma{H}$ is $t$-cancellative then by removing from it at most $1+\lf\fr{t}{2}\rf$ edges one gets a subhypergraph of $\ma{H}$, that is $\lf\fr{t}{2}\rf$-cover-free;
  \item [$(b)$] If $\ma{H}$ is $t$-cover-free then it is also $t$-union-free; if $\ma{H}$ is $t$-union-free then it is also $(t-1)$-cover-free.
\end{itemize}

\noindent These observations follow straightforwardly from the definitions above (see Theorem 3.2 of \cite{Furedi-canc} for the proof of the second claim  of $(a)$).

Below we assume that $r,t$ are fixed integers and $n$ tends to infinity. Let $F_t(n,r),C_t(n,r)$ and $U_t(n,r)$ denote the maximum size of $t$-cover-free,  $t$-cancellative and $t$-union-free $r$-graphs on $n$ vertices, respectively. Frankl and F\"uredi \cite{Frankl-Furedi-3} showed that for fixed $r,t$, it holds that

\begin{equation}\label{cff-frankl-furedi}
    F_t(n,r)=\big(\gamma(r,t)+o(1)\big)n^{\lc\fr{r}{t}\rc},
\end{equation}

\noindent where $\gamma(r,t)$ is a constant depending only on $r,t$ and $o(1)\rightarrow 0$ as $n\rightarrow\infty$. Then by observations  $(a)$, $(b)$, and \eqref{cff-frankl-furedi} (see (10) in \cite{Furedi-grids} and (4.2) in \cite{Furedi-canc}) we have

\begin{equation}\label{eq-cff-bound}
  \Omega(n^{\lc\fr{r}{t+1}\rc})=C_t(n,r)=O(n^{\lc\fr{r}{\lf t/2\rf}\rc})\text{\quad and\quad} \Omega(n^{\lc\fr{r}{t}\rc})=U_t(n,r)=O(n^{\lc\fr{r}{t-1}\rc}).
\end{equation}

\noindent In this paper we only study the exponents of $n$ in $C_t(n,r)$ and $U_t(n,r)$.
If $C_t(n,r)=\Theta(n^{\alpha})$ for some  fixed  number $\alpha>0$, then $\alpha$ is called the  {\it Tur\'an exponent} of $C_t(n,r)$.
The Tur\'an exponent of  $U_t(n,r)$ is defined similarly.
 Clearly, in \eqref{eq-cff-bound}, the upper and lower bounds on the exponents of $n$ are far from being tight, and in this paper we take another step towards bridging them.

In the literature, there are only a handful of results that improve on \eqref{eq-cff-bound}. For cancellative hypergraphs, it is known that $\fr{0.28}{2^r}\binom{n}{r}<C_1(n,r)\le\fr{2^r}{\binom{2r}{r}}\binom{n}{r}$ (see \cite{Frankl-Furedi-1} for the upper bound and \cite{tolhuizen2000new} for the lower bound), and \cite{Furedi-canc} showed that $\Omega(n^{\lf\fr{r}{2}\rf})=C_2(n,r)=O(n^{\lc\fr{r}{2}\rc})$ and $n^{2-o(1)}<C_2(n,3)=O(n^2)$ . For union-free hypergraphs, it is known that $U_2(n,r)=\Theta(n^{\lc 4r/3\rc/2})$ \cite{Frankl-Furedi-2} and $n^{2-o(1)}<U_r(n,r)=O(n^2)$ \cite{Furedi-grids}.

By the discussion above, it is clear that the Tur\'an exponents of $C_t(n,r)$ are known for $t=1$ and $t=2,~2\mid r$, and  the Tur\'an exponents of $U_t(n,r)$ are only known for $t=2$. In this paper we present new constructions which considerably narrow the gaps between the upper and lower bounds of \eqref{eq-cff-bound}, as stated next.

\subsection{$t$-cancellative $r$-graphs}

\begin{theorem}\label{theorem-canc}
For fixed integers $r\ge 3, t\ge 3$ and $n\rightarrow\infty$, it holds that

\begin{equation}\label{eq-canc}
 \Omega(n^{\lf\frac{2r}{t+2}\rf+\frac{2r\pmod{t+2}}{t+1}})=C_t(n,r)=O(n^{\lceil\frac{r}{\lfloor t/2\rfloor+1}\rceil}).
\end{equation}

\noindent Moreover, if $\gcd(2r-\lc\fr{2r-t-1}{t+2}\rc,t+1)=1$, then $C_t(n,r)=\Omega\big(n^{\lf\frac{2r}{t+2}\rf+\frac{2r\pmod{t+2}}{t+1}}(\log n)^{\frac{1}{t+1}}\big)$.
\end{theorem}

Note that by \eqref{eq-canc} for even $t$ and $r$ divisible by $t/2+1$ the  Tur\'an exponent of $C_t(n,r)$ is known.  We speculate that for any  $t,r\geq 3$ the upper bound of \eqref{eq-canc} can be further improved to $C_t(n,r)=O(n^{\lceil\frac{2r}{t+2}\rceil})$. In particular, it would be interesting to determine whether $C_3(n,5)=O(n^2)$. If our guess is correct, then the lower bound of \eqref{eq-canc} gives the Tur\'an exponent of $C_t(n,r)$ for all $r,t$ satisfying $(t+2)\mid 2r$.

For $(t+2)\nmid 2r$, it may be a difficult problem to determine the asymptotic order of $C_t(n,r)$. For example, F\"uredi \cite{Furedi-canc} conjectured that (see Conjecture 12.1 of \cite{Furedi-canc})
\begin{equation}\label{conj-canc}
    n^{k+1-o(1)}<C_2(n,2k+1)=o(n^{k+1})
\end{equation}
for all fixed $k\ge 1$. Moreover, he showed that for $k=1$, the conjectured upper bound $C_2(n,3)=o(n^2)$ is closely related to a celebrated and longstanding conjecture of Brown, Erd\H{o}s and S\'os \cite{Brown-Erdos-Sos-1971} in extremal graph theory (see Section \ref{section-sparse} below for the details). We are able to verify the lower bound part of F\"uredi's conjecture.

\begin{theorem}\label{2-cancellative}
  For any fixed integer $k\ge 1$ and $n\rightarrow\infty$, it holds that $C_2(n,2k+1)>n^{k+1-o(1)}$.
\end{theorem}

\subsection{$t$-union-free $r$-graphs}

\begin{theorem}\label{theorem-uf}
For fixed integers $r\ge3, t\ge 3$ and $n\rightarrow\infty$, we have that

\begin{equation}\label{eq-uf}
 \Omega(n^{\fr{r}{t-1}})=U_t(n,r)=O(n^{\lc\fr{r}{t-1}\rc}).
\end{equation}

\noindent Moreover, if $\gcd(r,t-1)=1$, then $U_t(n,r)=\Omega\big(n^{\fr{r}{t-1}}(\log n)^{\frac{1}{t-1}}\big)$.
\end{theorem}

Note that the upper bound was proved already in \eqref{eq-cff-bound}, whereas the lower bound   $U_t(n,r)=\Omega(n^{\fr{r}{t-1}})$, although not  explicitly stated, was implied by an earlier work of Blackburn (see Theorem 5 in \cite{blackburn}),  on  a problem in multimedia fingerprinting.
The main novelty of Theorem \ref{theorem-uf} compared to  \cite{blackburn}  is the application of sparse hypergraphs in proving this result (which gives a simpler proof), and the  improved lower bound of  $U_t(n,r)=\Omega\big(n^{\fr{r}{t-1}}(\log n)^{\frac{1}{t-1}}\big)$.

Clearly, \eqref{eq-uf} gives the Tur\'an exponent of $U_t(n,r)$ when $(t-1)\mid r$. Similarly, for $(t-1)\nmid r$, the determination of the correct order of $U_t(n,r)$ may be difficult. Indeed, F\"uredi and Ruszink\'o
\cite{Furedi-grids} conjectured that
\begin{equation}\label{conj-un}
    U_r(n,r)=o(n^2)
\end{equation}

\noindent for all fixed $r\ge 3$, where the conjectured upper bound is also related to the aforementioned conjecture of \cite{Brown-Erdos-Sos-1971}.

Recall that as mentioned previously, prior to this paper and \cite{blackburn}, for $t\ge 3$ the only known lower bound of $C_t(n,r)$ (resp. $U_t(n,r)$) was derived from that of $F_t(n,r)$, using \eqref{cff-frankl-furedi} and observation $(a)$ (resp. $(b)$).
%The novel idea of this paper is that,
Surprisingly, we show that {\it sparse hypergraphs}, as introduced in the next section, can be used as a unified tool to construct $t$-cancellative and $t$-union-free $r$-graphs (see Lemmas \ref{lemma-0}, \ref{lemma-10} and \ref{lemma-111} below for details), thereby providing much better lower bounds for $C_t(n,r)$ and $U_t(n,r)$. % than the previous constructions.

The remaining part of this paper is organised as follows. In Section \ref{section-sparse} we introduce sparse hypergraphs and in Section \ref{connections} we show how they can be used to prove the lower bounds in Theorems \ref{theorem-canc}, \ref{2-cancellative} and \ref{theorem-uf}. We defer the proof of the upper bound in Theorem \ref{theorem-canc} to Section \ref{upper}. %Finally, in Section \ref{concluding} we discuss the lower and upper bounds of $C_3(n,r)$ for $r=3,4,5$, and present some open questions for future research.

{\it Notations.} We will use the standard Bachmann-Landau notations $\Omega(\cdot),\Theta(\cdot),O(\cdot)$ and $o(\cdot)$, whenever the constants are not important. All logarithms are of base 2.

\section{Constructions based on  sparse hypergraphs}\label{section-sparse}
%\noindent
For integers $v\ge r+1, e\ge 2$, an $r$-graph is called {\it $\mr{G}_r(v,e)$-free} if the union of any $e$ distinct edges contains at least $v+1$ vertices. Such $r$-graphs are also called {\it sparse} due to the sparsity of its edges, i.e., any set of $v$ vertices spans less than $e$ edges. Let $f_r(n,v,e)$ denote the maximum number of edges of a $\mr{G}_r(v,e)$-free $r$-graph on $n$ vertices. %The following classical result was proved by
Brown, Erd\H{o}s and S\'os \cite{Brown-Erdos-Sos-1971} showed that
\begin{equation*}\label{BES}
\Omega(n^{\frac{er-v}{e-1}})=f_r(n,v,e)=O(n^{\lceil\frac{er-v}{e-1}\rceil}),
\end{equation*}

\noindent where the lower bound was proved by a probabilistic construction using the {\it alteration method} (see, e.g. \cite{alon2016probabilistic}).
Indeed, a slightly more sophisticated probabilistic analysis leads to the following stronger lower bound. %on $f_r(n,v,e)$.

\begin{lemma}[see Lemma 1.7 of \cite{Furedi-grids} and Proposition 6 of \cite{shangguan2019universally}]\label{probability}
    Let $s\ge 1, r\ge 3$ and $(v_i,e_i)$ for $1\le i\le s$ be pairs of  integers satisfying $v_i\ge r+1,e_i\ge 2$.
\begin{itemize}
    \item [$(a)$]  Let $h:=\min~\{\frac{e_ir-v_i}{e_i-1}:1\le i\le s\}$. Then there exists an $r$-graph with $\Omega(n^h)$ edges that is simultaneously $\mr{G}_r(v_i,e_i)$-free for each $1\le i\le s$.

    \item [$(b)$] Suppose further that $e_1\ge 3, \gcd(e_1-1,e_1r-v_1)=1$ and $\frac{e_1r-v_1}{e_1-1}<\frac{e_ir-v_i}{e_i-1}$ for each $2\le i\le s$.
    Then there exists an $r$-graph with $\Omega(n^{\frac{e_1r-v_1}{e_1-1}}(\log n)^{\frac{1}{e_1-1}})$ edges that is simultaneously $\mr{G}_r(v_i,e_i)$-free for each $1\le i\le s$.
\end{itemize}
\end{lemma}

In the literature, the following celebrated conjecture is well-known (see, e.g. \cite{Brown-Erdos-Sos-1971} and \cite{AlonShapira}).

\begin{conjecture}\label{hard}
For all fixed integers $r>k\ge 2,e\ge 3$ and $n\to\infty$, it holds that $f_r(n,er-(e-1)k+1,e)=o(n^k)$.
\end{conjecture}

In \cite{Furedi-canc} and \cite{Furedi-grids} it was shown that $C_2(n,3)=\Theta\big(f_3(n,7,4)\big)$ and $U_r(n,r)=O\big(f_r(n,r^2-r+1,r+1)\big)$, respectively. Indeed, those two bounds and Conjecture \ref{hard} motivated the two conjectured upper bounds in \eqref{conj-canc} and \eqref{conj-un}. It is to our surprise that sparse hypergraphs are even more closely related to $t$-cancellative and $t$-union-free $r$-graphs, than what was known before. In particular, we prove that $t$-cancellative and $t$-union-free $r$-graphs can both be  constructed by sparse hypergraphs, as stated in the following three lemmas.

\begin{lemma}\label{lemma-0}
  Let $\ma{H}$ be an $r$-graph that is simultaneously $\mr{G}_r(tr+\lc\fr{2r-t-1}{t+2}\rc,t+2)$-free and $\mr{G}_r(2r-\lc\fr{2r-t-1}{t+2}\rc-1,2)$-free, then it must be $t$-cancellative.
\end{lemma}

\begin{lemma}\label{lemma-10}
  If a $(2k+1)$-graph is $\mr{G}_r(4k+2,3)$-free, then it is also 2-cancellative.
\end{lemma}

Note that an $r$-graph $\ma{H}$ is said to be {\it $r$-partite} if there exists a partition of its vertices into pairwise disjoint subsets $V_1,\ldots,V_r$ such that for any $A\in\ma{H}$ and $1\le i\le r$, it holds that $|A\cap V_i|=1$.

\begin{lemma}\label{lemma-111}
  If an $r$-partite $r$-graph $\ma{H}$ is simultaneously $\mr{G}_r(tr-r,t)$-free and $\mr{G}_r(tr,2t)$-free, then it is also $t$-union-free.
\end{lemma}

Using the above lemmas in concert with Lemma \ref{probability}, one can easily deduce the lower bounds in Theorems \ref{theorem-canc}, \ref{2-cancellative} and \ref{theorem-uf}, as detailed below.

\subsection{Proofs of Theorems \ref{theorem-canc}, \ref{2-cancellative} and \ref{theorem-uf}}\label{section-3}

\begin{proof}[\textbf{Proof of Theorem \ref{theorem-canc}}]
The proof of the upper bound is  postponed to Section \ref{upper}. For the first lower bound we invoke Lemma \ref{probability} $(a)$. Indeed, since for $r\ge 3,t\ge 3$, $\frac{2r-\lc\fr{2r-t-1}{t+2}\rc}{t+1}\le \lc\fr{2r-t-1}{t+2}\rc+1$ and
$\frac{2r-\lc\fr{2r-t-1}{t+2}\rc}{t+1}=\lf\frac{2r}{t+2}\rf+\frac{2r \mod (t+2)}{t+1}:=h$, then there exists an $r$-graph with $\Omega(n^h)$ edges that is simultaneously $\mr{G}_r(tr+\lc\fr{2r-t-1}{t+2}\rc,t+2)$-free and $\mr{G}_r(2r-\lc\fr{2r-t-1}{t+2}\rc-1,2)$-free. Then, by  Lemma \ref{lemma-0} this $r$-graph is $t$-cancellative.

For the second lower bound we invoke  Lemma \ref{probability} (b). Indeed,  if $\gcd(2r-\lc\fr{2r-t-1}{t+2}\rc,t+1)=1$, then it is not hard to see that $t+2\nmid 2r-t-1$ and hence $\frac{2r-\lc\fr{2r-t-1}{t+2}\rc}{t+1}<\lc\fr{2r-t-1}{t+2}\rc+1.$  Lastly, we invoke again Lemma \ref{lemma-0}. %The second lower bound can be proved similarly by applying Lemma \ref{probability} $(b)$ and Lemma \ref{lemma-0}.
\end{proof}

\begin{proof}[\textbf{Proof of Theorem \ref{2-cancellative}}]
  It was shown by Alon and Shapira (see Theorem 1 of \cite{AlonShapira}) that for any fixed $k\ge 1$, $$f_{2k+1}(n,4k+2,3)>n^{k+1-o(1)},$$ where $o(1)\rightarrow 0$ as $n\rightarrow\infty$.
  The result follows by invoking Lemma \ref{lemma-10}.
\end{proof}

To prove Theorem \ref{theorem-uf} we need a well-known result of Erd\H{o}s and Kleitman \cite{Erdos-Kleitman}.

\begin{lemma}[see Theorem 1 of \cite{Erdos-Kleitman}]\label{Erdos-Kleitman}
  Any $r$-graph $\ma{H}$ contains   an $r$-partite subhypergraph $\ma{H}^*\subseteq\ma{H}$ with at least  $|\ma{H}^*|\ge\frac{r!}{r^r}|\ma{H}|$ edges.
\end{lemma}

\begin{proof}[\textbf{Proof of Theorem \ref{theorem-uf}}]
 For the first lower bound, observe that for $r,t\ge 3$, $\frac{r}{t-1}<\frac{tr}{2t-1}$, so by Lemma \ref{probability} $(a)$ and Lemma \ref{Erdos-Kleitman} there exists an $r$-partite $r$-graph with $\Omega(n^{\frac{r}{t-1}})$ edges that is simultaneously $\mr{G}_r(tr-r,t)$-free and $\mr{G}_r(tr,2t)$-free. Then by Lemma \ref{lemma-111} this $r$-graph is $t$-union free. The second lower bound is proved similarly by invoking Lemma \ref{probability} $(b)$, Lemma \ref{Erdos-Kleitman} and Lemma \ref{lemma-111}.
\end{proof}

The proofs of Lemmas \ref{lemma-0}, \ref{lemma-10} and \ref{lemma-111} are presented in the next section.

\section{Proofs of Lemmas \ref{lemma-0}, \ref{lemma-10} and \ref{lemma-111}}\label{connections}

\subsection{Proof of Lemma \ref{lemma-0}}

We prove the following claim which is slightly more general than Lemma \ref{lemma-0}. The proof of the latter  follows by setting $x=\lceil \frac{2r-t-1}{t+2}\rceil$.
%The proof of Le where Instead of proving Lemma \ref{lemma-0}, we in fact prove the following more general result.

\begin{claim}
 Let  $\ma{H}$ be an $r$-graph that is simultaneously $\mr{G}_r(tr+x,t+2)$-free and $\mr{G}_r(2r-x-1,2)$-free for some integer  $0\le x\le r-1$, then $\ma{H}$ is  $t$-cancellative.
\end{claim}

\begin{proof}
  Assume towards   contradiction that there exist $t+2$ distinct edges $A_1,\ldots,A_t,B,C$ of $\ma{H}$ such that $(\cup_{i=1}^t A_i)\cup B=(\cup_{i=1}^t A_i)\cup C$, hence
  \begin{equation}\label{eq-2}
   (B\cup C)\backslash (B\cap C) \subseteq \cup_{i=1}^t A_i.
  \end{equation}
  Furthermore, since $\ma{H}$ is $\mr{G}_r(2r-x-1,2)$-free then
  \begin{equation}\label{eq-21}
     |B\cap C|\leq x.
  \end{equation}
  %\noindent In order to obtain a contradiction, we proceed to show that $|(\cup_{i=1}^t A_i)\cup B\cup C|\le tr+x$.
 Therefore \begin{align}
    |(\cup_{i=1}^t A_i)\cup B\cup C|&=|\cup_{i=1}^t A_i| +| B\cup C|-|(\cup_{i=1}^t A_i)\cap (B\cup C)|\nonumber\\
    &\leq tr +| B\cup C|-|(\cup_{i=1}^t A_i)\cap (B\cup C)|\nonumber\\
    &\leq tr +| B\cup C|-(| B\cup C|- |(B\cap C)|)\label{stam}\\
    &\leq tr +x,\label{stam2}
  \end{align}
  where \eqref{stam} and \eqref{stam2} follow from \eqref{eq-2} and \eqref{eq-21} respectively. The contradiction follows since  $\ma{H}$ is  $\mr{G}_r(tr+x,t+2)$-free.
\end{proof}

%Applying lemma \ref{lemma-2} with $x=\lf\fr{2r-t-1}{t+2}\rf$ gives lemma \ref{lemma-0}.
%the following proposition.

\subsection{Proof of Lemma \ref{lemma-10}}

%\begin{proof}[\textbf{Proof of Lemma \ref{lemma-10}}]
%\noindent
Let $\ma{H}$ be a $(2k+1)$-graph that is $\mr{G}_{2k+1}(4k+2,3)$-free,  we show next that $\ma{H}$ is also 2-cancellative. Suppose for contradiction that there exist four distinct edges $A_1,A_2,B,C\in\ma{H}$ such that $A_1\cup A_2\cup B=A_1\cup A_2\cup C$. Let  $|B\cap C|=x$  for some  integer $0\le x\le 2k$. It is easy to see that $|(B\cup C)\setminus(B\cap C)|=2\times(2k+1-x)$ and $(B\cup C)\setminus(B\cap C)\subseteq A_1\cup A_2$. Therefore, by the pigeonhole principle, there exists $i\in\{1,2\}$ such that $|A_i\cap \big((B\cup C)\setminus(B\cap C)\big)|\ge 2k+1-x$, which implies that $|A_i\setminus (B\cup C)|\le x$. Then
  $$|A_i\cup B\cup C|=|B\cup C|+|A_i\setminus (B\cup C)|=2(2k+1)-x+|A_i\setminus (B\cup C)|\le 4k+2,$$ which violates the assumption that $\ma{H}$ is $\mr{G}_{2k+1}(4k+2,3)$-free.
%\end{proof}

\subsection{Proof of Lemma \ref{lemma-111}}

The following fact is easy to verify.
\begin{fact}\label{lemma-1}
 If an $r$-graph $\ma{H}$ with at least $t$ edges is $\mr{G}_r(tr-r,t)$-free, then it is also $\mr{G}_r(sr-r,s)$-free for any $1\le s\le t$.
\end{fact}

%To prove Lemma \ref{lemma-111} we will need the following two lemmas.

%\begin{lemma}\label{lemma-1}
%  If an $r$-graph $\ma{H}$ with at least $t$ edges is $\mr{G}_r(tr-r,t)$-free, then it is in fact $\mr{G}_r(sr-r,s)$-free for any $1\le s\le t$.
%\end{lemma}

%\begin{proof}
  %Suppose for contradiction that $\ma{H}$ is not $\mr{G}_r(sr-r,s)$-free for some $1\le s\le t-1$. Then there exist $s$ distinct edges $A_1,\ldots,A_s$ of $\ma{H}$ such that $\cup_{i=1}^s A_i\le sr-r$. Choosing arbitrary $t-s$ distinct edges $A_{s+1},\ldots,A_t\in\ma{H}\setminus\{A_1,\ldots,A_s\}$, it follows that $$|\cup_{i=1}^t A_i|\le |\cup_{i=1}^s A_i|+(t-s)r\le tr-r,$$ violating the $\mr{G}_r(tr-r,t)$-freeness of $\ma{H}$.
%\end{proof}

We need one more result to prove Lemma \ref{lemma-111}.

\begin{lemma}\label{lemma-2}
  Let  $\ma{H}$ be an $r$-partite $r$-graph  with vertex parts  $V_1,\ldots,V_r$. %$V(\ma{H})=V_1\cup\cdots\cup V_r$,
  If $\ma{H}$ is $\mr{G}_r(tr-r,t)$-free, then for any $t$ distinct edges $A_1,\ldots,A_t$ of $\ma{H}$, there exists $1\le i\le r$, such that the $t$ vertices $A_1\cap V_i,\ldots,A_t\cap V_i$ are all distinct.
\end{lemma}

\begin{proof}
  %Let $\ma{H}$ be any $r$-graph satisfying the required property.
  Suppose for contradiction that the $t$ distinct edges $A_1,\ldots,A_t\in\ma{H}$ violate the assertion of the lemma. Then, for each $1\le i\le r,  |\cup_j (A_j\cap V_i)|\leq t-1$, hence
  $|\cup_j A_j|=\sum_{i=1}^r|\cup_j (A_j\cap V_i)|\leq tr-r$, and we arrive at a contradiction.
\end{proof}

%Next we present the proof of Lemma \ref{lemma-111}.

\begin{proof}[\textbf{Proof of Lemma \ref{lemma-111}}]
  %Let $\ma{H}$ be an $r$-partite $r$-graph, that is also $(tr-r,t)$-free and $(tr,2t)$-free.
  Suppose for contradiction that $\ma{H}$ is not $t$-union-free. Then there exist distinct subhypergraphs $\ma{A},\ma{B}\s\ma{H}, 1\le|\ma{A}|,|\ma{B}|\le t$, such that $\cup_{A\in\ma{A}}A=\cup_{B\in\ma{B}} B$.

  It is easy to see that $|\ma{A}|=|\ma{B}|=t$. Indeed, assume that $|\ma{A}|=s-1\le t-1$ and let $\ma{A}=\{A_1,\ldots,A_{s-1}\}$. Since $B\s \cup_{i=1}^{s-1} A_i$  for  $B\in\ma{B}$, then
  $$|\cup_{i=1}^{s-1} A_i\cup B|= |\cup_{i=1}^{s-1} A_i|\le (s-1)r,$$
  which by Fact \ref{lemma-1}, violates the $\mr{G}_r(sr-r,s)$-freeness of $\ma{H}$.

  It is also not hard to show that $\ma{A}\cap\ma{B}\neq\emptyset$, since otherwise $\ma{A}\cup\ma{B}$ consists of $2t$ distinct edges of $\ma{H}$ that satisfy
  $$|(\cup_{A\in\ma{A}} A)\cup(\cup_{B\in\ma{B}} B)|=|\cup_{A\in\ma{A}} A|\le tr,$$
  violating the $\mr{G}_r(tr,2t)$-freeness of $\ma{H}$. Therefore we may assume that $|\ma{A}\cap\ma{B}|=i\ge 1$.

  Suppose that we have $$\ma{A}:=\{C_1,\ldots,C_i,A_{i+1},\ldots,A_t\}\text{ and }\ma{B}:=\{C_1,\ldots,C_i,B_{i+1},\ldots,B_t\},$$ where $C_1,\ldots,C_i,A_{i+1},\ldots,A_t,B_{i+1},\ldots,B_t$ are
  %$\ma{A}\cup\ma{B}$ is a subset of
  $2t-i$ distinct edges of $\ma{H}$.
  By applying Lemma \ref{lemma-2} to the $t$ edges in  $\ma{A}':=\ma{A}\cup \{B_{i+1}\}\backslash \{C_1\}$ we conclude that there exists  $1\le j\le r$ such that the $t$  elements $A\cap V_j, A\in \ma{A}'$  are pairwise distinct.

  Let $c_l=C_l\cap V_j, \text{ for } 1\le l\le i \text{ and }
    a_l=A_l\cap V_j,  b_l=B_l\cap V_j, \text{ for } i< l\le t $,  and note that $b_{i+1}, c_2,\ldots ,c_i,a_{i+1},\ldots,a_t$ are pairwise distinct vertices. By assumption %$\cup_{A\in\ma{A}}A=\cup_{B\in\ma{B}} B$, in particular,
  \begin{equation}\label{eq-1}
    (\cup_{A\in\ma{A}}A)\cap V_j=\{c_1,\ldots,c_i,a_{i+1},\ldots,a_t\}=\{c_1,\ldots,c_i,b_{i+1},\ldots,b_t\}=(\cup_{B\in\ma{B}} B)\cap V_j.
  \end{equation}

  \noindent %However, since $c_1,\ldots,c_i,a_{i+1},\ldots,a_t,b_{i+1},\ldots,b_{2i}$ are pairwise distinct, the left hand side of \eqref{eq-1}
Hence,  by \eqref{eq-1} and the pairwise distinctness of  $b_{i+1}, c_2,\ldots, c_i,a_{i+1},\ldots,a_t$, it is clear that $b_{i+1}=c_1$ and therefore also
  $c_1,\ldots,c_i,a_{i+1},\ldots,a_t$ are pairwise distinct too.

  %,\ldots,b_{2i}\}\s\{c_1,\ldots,c_i,a_{i+1},\ldots,a_t\}$. Since $a_{i+1},\ldots,a_{t},b_{i+1},\ldots,b_{2i}$ are pairwise distinct, it follows that %$\{b_{i+1},\ldots,b_{2i}\}\s\{c_1,\ldots,c_i,a_{i+1},\ldots,a_t\}$ (as implied by \eqref{eq-1}) holds only if
  %$$\{b_{i+1},\ldots,b_{2i}\}=\{c_1,\ldots,c_i\},$$

  %\noindent
  %which further implies that  Therefore,
  We conclude that the leftmost set  of \eqref{eq-1} consists of exactly $t$ distinct elements, whereas the rightmost set  consists of at most $t-1$ distinct elements, and we arrive at a  contradiction.   \end{proof}

  %The case of $t<2i$ can be proved analogously as follows. Write $t+x=2i$ for some integer $x\ge 1$. Since $i<t$, it follows that $x<i$. Let us apply Lemma \ref{lemma-2} to the $t$ edges $A_{i+1},\ldots,A_t,B_{i+1},\ldots,B_{t},C_1,\ldots,C_x$. Then there exists $1\le j\le r$ such that the intersections of $V_j$ and the above $t$ edges are pairwise distinct. With the notation of \eqref{eq-0},  it is easy to see that $a_{i+1},\ldots,a_t,b_{i+1},\ldots,b_t,c_1,\ldots,c_x$ are pairwise distinct. As \eqref{eq-1} holds also in this case, it is not hard to check that $\{b_{i+1},\ldots,b_{t}\}\s\{c_1,\ldots,c_i,a_{i+1},\ldots,a_t\}$ holds only if
  %$$\{b_{i+1},\ldots,b_{t}\}=\{c_{x+1},\ldots,c_i\}$$
  %(note that $t-i=i-x$), which further implies that $c_1,\ldots,c_i,a_{i+1},\ldots,a_t$ are pairwise distinct.
  %Again, we conclude that the subset on the left hand side of \eqref{eq-1} consists of exactly $t$ distinct elements, whereas the subset on the right hand side of \eqref{eq-1} consists of at most $t-(i-x)<t$ distinct elements, and we arrive at a contradiction, completing the proof of Lemma \ref{lemma-111}.

\section{Proof of Theorem \ref{theorem-canc}, the upper bound}\label{upper}
%\noindent %Our goal in this subsection is to show that $C_t(n,r)=O(n^{\lceil\frac{2r}{t+2}\rceil})$.
We will need several auxiliary lemmas before presenting the proof that $C_t(n,r)=O(n^{\lc\frac{r}{\lf t/2\rf+1}\rc})$. The  following fact is easy to verify.

\begin{fact}\label{fact-1}
 A $t$-cancellative $r$-graph with at least $t+2$ edges  is also $s$-cancellative for any $1\le s\le t$.
\end{fact}

For an $r$-graph $\ma{H}\s\binom{[n]}{r}$ and a subset $T\s[n]$ with $1\le |T|\le r$, the {\it codegree} of $T$ with respect to $\ma{H}$, denoted as $\deg_{\ma{H}}(T)$, is the number of edges of $\ma{H}$ that contain $T$, i.e.,  $\deg_{\ma{H}}(T)=|\{A\in\ma{H}:T\s A\}|.$

\begin{lemma}\label{naivelemma2}
  For integers $s\ge 1$ and $1\le k\le r$, any $r$-graph $\mathcal{H}$ contains a subhypergraph %$\mathcal{H}'\subseteq\mathcal{H}$
with at least $\max\{|\mathcal{H}|-s\binom{n}{k},0\}$ edges, such that each $k$-subset has codegree either zero or at least $s+1$.
  %can be made to have no $k$-subset of codegree less than $s+1$ by deleting at most $s\binom{n}{k}$ of its edges.
\end{lemma}

\begin{proof}
  Let us successively remove from $\mathcal{H}$ the edges that contain at least one $k$-subset of codegree at most $s$. Let $A_i$ be the $i$-th removed edge of $\mathcal{H}$, and $T_i$ be some $k$-subset of codegree at most $s$ contained in $A_i$. We say that $T_i$ is {\it responsible} for $A_i$. During the edge removal process, the codegree (with respect to the $r$-graph after the removal of the edges $A_1,A_2,\ldots$) of any $k$-subset can only decrease, therefore each of the $\binom{n}{k}$  $k$-subsets  can be responsible for the removal of at  most $s$ edges. Hence,  the process will terminate  after at most $s\binom{n}{k}$ edge removals. Clearly, the  resulting $r$-graph is either empty or has at least $|\mathcal{H}|-s\binom{n}{k}$ edges, and it satisfiess the assertion on the codegrees.
\end{proof}

\begin{lemma}\label{last}
%Let $t,r$ be fixed positive integers.
For any positive integer $t$,  $C_{2t}(n,r)=O(n^{\lceil\frac{r}{t+1}\rceil})$.
\end{lemma}

\begin{proof}
It is sufficient to prove the lemma under the assumption that $(t+1)\mid r$. Indeed, assuming that $C_{2t}(n,(t+1)k)=O(n^k)$, then for $r$ not divisible by $t+1$ write  $r=(t_1+1)(k-1)+y$, with  $1\le y\le t_1$ an integer. Then,
$$C_{2t}(n,r)\le C_{2t}(n+t+1-y,r+t+1-y)=C_{2t}(n+t+1-y,(t+1)k)=O(n^{k})=O(n^{\lceil\frac{r}{t+1}\rceil}),$$ as needed.

Next we   show that $C_{2t}(n,(t+1)k)=O(n^k)$. Let $\mathcal{H}$ be any $2t$-cancellative $(t+1)k$-graph defined on $n$ vertices. We proceed to show that $|\mathcal{H}|\le 2\binom{n}{k}+2t+1$. Assume towards  contradiction that $|\mathcal{H}|\ge 2\binom{n}{k}+2t+2$, then  by applying Lemma \ref{naivelemma2} with $s=2$  there exists a subhypergraph $\mathcal{H}'\subseteq\mathcal{H}$ with at least  $|\mathcal{H}'|\ge 2t+2$ edges, such that no $k$-subset has   codegree equal to one or two.

Let $C\in \ma{H}'$ be an arbitrary edge and $X\s C$ an arbitrary $k$-subset. By the codegree assertion, there exist an edge $D\neq C$ that contains $X$.
Let $X_1,\ldots, X_t$ and $Y_1, \ldots, Y_t$  be  arbitrary partitions of $C\backslash X$ and $D\backslash X$ respectively, to pairwise disjoint $k$-subsets.
Again, by the codegree assertion, for each $i$ there exist edges $A_i,B_i \in \ma{H}'\backslash\{C,D\}$  such that
$X_i\s A_i$ and $Y_i\s B_i$. Consequently,
$$A_1\cup\cdots\cup A_{t}\cup B_1\cdots\cup B_{t}\cup C=A_1\cup\cdots\cup A_{t}\cup B_1\cdots\cup B_{t}\cup D,$$
and we arrive at a contradiction to the fact that $\ma{H}$ is  $2t$-cancellative   combined with  Fact \ref{fact-1}.\end{proof}

%of some edge $C\in \ma{H}'$. By
%In particular, there exist distinct $C,D\in\mathcal{H}'$ such that there is an $k$-subset $X\subseteq (C\cap D)$. Let us partition $C\setminus X$ (resp. $D\setminus X$) into the union of $t_1$ pairwise disjoint $k$-subsets. Assume without loss of generality that  $\{Y_1,\ldots,Y_{t_1}\}$ (resp. $\{Z_1,\ldots,Z_{t_1}\}$) is such a partition of $C\setminus X$ (resp. $D\setminus X$). Since for each $1\le i\le t_1$, the codegrees of $Y_i$ and $Z_i$ are both at least three, there exist $A_i,B_i\in\mathcal{H}'\setminus\{C,D\}$ such that $Y_i\subseteq A_i$ and $Z_i\subseteq B_i$.
%Therefore it is easy to see that
%If the $2t$ subsets $A_1,\ldots,A_{t_1},B_1,\ldots,B_{t_1}$ are pairwise distinct, we immediately arrive at a contradiction. Otherwise, since $|\mathcal{H}'|\ge t+2$, we can easily make $\mathcal{H}'$ violate the $t$-cancellativeness property by adding an appropriate number of edges simultaneously to both hand sides of the above equality.

\begin{proof}[\textbf{Proof of Theorem \ref{theorem-canc}, the upper bound}] For even $t$, the result follows from  Lemma \ref{last}. For odd $t$
$$C_{t}(n,r)\leq C_{t-1}(n,r)=
O(n^{\lceil \frac{r}{\frac{t-1}{2}+1}\rceil })=
O(n^{\lceil \frac{r}{\lfloor t/2\rfloor +1}\rceil }).$$  \end{proof}
%$t$ is odd, then we apply Lemma \ref{last} with $t-1$.

\section*{Acknowledgements}

\noindent The research of Chong Shangguan and Itzhak Tamo was supported by ISF grant No. 1030/15 and NSF-BSF grant No. 2015814.

%\section{New bounds of $C_3(n,r)$ for $r=3,4,5$}\label{concluding}

%\begin{theorem}\label{r=3}
%For $n\to\infty$, it holds that $n^{2-o(1)}<C_3(n,3)=o(n^2)$.
%\end{theorem}

%\begin{theorem}\label{r=4}
%For $n\to\infty$, it holds that $n^{2-o(1)}<C_3(n,3)=o(n^2)$.
%\end{theorem}

%\begin{theorem}\label{r=3}

%\end{theorem}

{\small
\bibliographystyle{plain}
\bibliography{sparse-1}
}
 \end{document}